\documentclass[10pt]{amsart}
\usepackage{amsmath,amssymb,amsthm}
\usepackage{graphicx}
\usepackage{color}
\newtheorem{theorem}{Theorem}    

\newtheorem{proposition}{Proposition} 
 
\newtheorem{question}{Question} 
\newtheorem{corollary}{Corollary}
\theoremstyle{definition}

\newtheorem*{remark*}{Remark}

\newcommand{\R}{\mathbb{R}}

\title{A note on knot fertility}
\author[T.Ito]{Tetsuya Ito}
\address{Department of Mathematics, Kyoto University, Kyoto 606-8502, JAPAN}
\email{tetitoh@math.kyoto-u.ac.jp}
 
\begin{document}

\begin{abstract}
We show that a knot whose minimum crossing number $c(K)$ is even and greater than $30$ is not fertile; there exists a knot $K'$ with crossing number less than $c$ such that $K'$ is not obtained from a minimum crossing number diagram of $K$ by suitably changing the over-under information.
\end{abstract}

\maketitle

A \emph{(knot) shadow} $S=S(D)$ of a knot diagram $D$ is an immersed circle (in $\R^{2}$ or $S^{2}$) obtained from $D$ by forgetting the over-under information at the crossings. A knot $K$ is \emph{supported} by a shadow $S$ if there is a diagram $D'$ of $K$ whose shadow is $S$. Namely, by suitably assigning over-under information at each crossing one can get $K$ from the shadow $S$. 

A knot $K$ is \emph{fertile} if for any prime knot $K'$ with $c(K') <c(K)$, there exists a minimum crossing number diagram $D$ of $K$ such that $S(D)$ supports $K'$. The notion of fertility was introduced in \cite{CH+}, where they observed that several low-crossing number knots $0_1, 3_1, 4_1, 5_2, 6_2, 6_3, 7_6$ are fertile. They also showed that no knots with crossing number 8--10 are fertile, and they asked the follwowing question.
\begin{question}
\label{question:main}
Is there a fertile knot with crossing number $>7$?
\end{question}

In \cite{Ha} Hanaki showed that there are no \emph{alternating} fertile knot with crossing number $>7$ by using Tait conjectures, the structure of minimum crossing number diagram of alternating knots.

For a knot $K$ we denote by $c(K)$ and $b(K)$ the minimum crossing number and the braid index. In this note we show the following non-existence result of fertile knots with large crossing numbers. 

\begin{theorem}
\label{theorem:main}
If the crossing number $c(K)$ of $K$ is even and $c(K)>30$, then $K$ is not fertile.
If the crossing number $c(K)$ of $K$ is odd, $b(K)\leq 3$ and $c(K)>25$, then $K$ is not fertile.
\end{theorem}

For a knot diagram $D$ (resp. shadow $S$), we denote by $c(D)$ and $s(D)$ (resp, $c(S)$ and $s(S)$) the number of crossings and Seifert circles, respectively.
The \emph{genus} of a diagram $D$ is the genus of Seifert surface of obtained by Seifert's algorithm, which is given by $g(D) = \frac{1}{2}(-s(D)+c(D)+1)$. 

We define the genus of a shadow $S$ by $g(S) = \frac{1}{2}(-s(S)+c(S)+1)$. 
Clearly, if $K$ is supported by $S$ then the genus of knot $K$ satisfies $g(K) \leq g(S)$. This simple observation leads to the following.

\begin{proposition}
\label{prop:s(D)}
Assume that $K$ is fertile.
\begin{enumerate}
\item If $c(K)$ is even, then $K$ admits a minimum crossing diagram $D$ with $s(D)\leq 3$ and $2g(K)-1 \leq \frac{c(K)}{2}$.
\item If $c(K)$ is odd, then $K$ admits a minimum crossing diagram $D$ with $s(D)\leq 4$ and $2g(K)-1 \leq \frac{c(K)-1}{2}$
\end{enumerate}
\end{proposition}
\begin{proof}
If $c=c(K)$ is odd, there is a minimum crossing diagram $D$ of $K$ such that its shadow $S(D)$ supports the $(2,c-1)$ torus knot $T_{2,c-1}$. Therefore
\[ -2+(c-1) = 2g(T_{2,c-1})-1 \leq 2g(S)-1 = -s(D)+c \]
Similarly, if $c$ is even, from the $(2,c-1)$ torus knot $T_{2,c-2}$ we get 
\[ -2+(c-2) = 2g(T_{2,c-1})-1 \leq 2g(S)-1 = -s(D)+c \]
Let $K_{c-1}$ be the twist knot with $c(K_{c-1})=c-1$, 
\[ b(K_{c-1}) = \begin{cases}\frac{c}{2} & c\mbox{:even}\\
\frac{c+1}{2} & c\mbox{:odd}
\end{cases}\]
Thus for a minimum crossing diagram $D$ whose shadow supports $K_{c-1}$
\[ 2g(K)-1 \leq 2g(D)-1=-s(D)+c \leq -b(K_{c-1})+c \leq 
\begin{cases}\frac{c}{2} & c\mbox{:even}\\
\frac{c(K)-1}{2} & c\mbox{:odd}
\end{cases}\]
 \end{proof}

\begin{corollary}
\label{cor:braid-index}
If a knot $K$ is fertile, 
$b(K) \leq 3$ if $c(K)$ is even, and $b(K)\leq 4$ if $c(K)$ is odd. 
\end{corollary}

Our proof of main theorem is based on the following observation on the crossing number and the genus of closed 3-braids which is interesting in its own right.
\begin{theorem}
\label{theorem:genus}
If $b(K)\leq 3$, then $-3 +\frac{3}{5}c(K) \leq 2g(K)-1$.
\end{theorem}
\begin{proof}
Let $\sigma_1,\sigma_2$ be the standard generator of the 3-strand braid group $B_3$ and let $\mathcal{B}=\{a_1=\sigma_{1,2} =\sigma_1, a_2=\sigma_{2,3}= \sigma_{2},a_3=\sigma_{1,3}=\sigma_{2}\sigma_1\sigma_{2}^{-1}\}$ be the band generators of $B_3$. A band generator $\sigma_{i,j}$ can be viewed as a boundary of a twisted band connecting $i$-th and $j$-th strand. Thus for a closed 3-braid representative $\beta$ of $K$, from a word representative of $\beta$ over the band generator $\mathcal{B}$ one can construct a Seifert surface of $K$ called a  Bennequin surface.

For $\beta \in B_3$ let $\ell_{\mathcal{B}}(\beta)$ be the length of $\beta$ with respect to the band generator $\mathcal{B}$. By Bennequin's theorem \cite{Be} (see also Birman-Menasco \cite{BM}) a closed 3-braid $K$ bounds a minimum genus Bennequin surface hence
\[
\label{eqn:Bennequin} 2g(K)-1 = \min\{-3+\ell_{\mathcal{B}}(\beta)\: | \: \beta \mbox{ represents } K \}
\]
Let $\beta_0$ be a 3-braid representative of $K$, written as a product of band generators that attains the minimum of $\ell_{\mathcal{B}}$ and $A_1,A_2,A_3$ be the number of letters $a_1^{\pm 1},a_2^{\pm 1},a_3^{\pm 1}$ in $\beta_0$.  Since $(\sigma_1\sigma_2)a_i(\sigma_1\sigma_2)^{-1} = a_{i+1}$ (here indices are taken mod $3$), by taking conjugates if necessary, we may assume that $A_3\leq A_1,A_2$. 

Since a closed braid $\beta_0$ can be seen as a diagram with crossings $A_1 + A_2 + 3 A_3$,
\[ A_1 + A_2+ 3A_3 \geq c(K). \]
Therefore 
\[ 2g(K)-1 = -3+\ell_{\mathcal{B}}(\beta)=-3+(A_1 + A_2 + A_3) \geq -3+ \frac{3}{5}c(K).\]
\end{proof}

\begin{proof}[Proof of Theorem \ref{theorem:main}]
If $c=c(K)$ is even, by Corollary \ref{cor:braid-index} $b(K) \leq 3$. Hence by Theorem \ref{theorem:genus} and Proposition \ref{prop:s(D)} 
\[ -3+ \frac{3}{5}c \leq 2g(K)-1 \leq \frac{c}{2} \iff c \leq 30 \] 
Similarly, if $c=c(K)$ is odd and $b(K)=3$, 
\[ -3+ \frac{3}{5}c\leq 2g(K)-1 \leq \frac{c-1}{2} \iff c \leq 25\] 
\end{proof}

Although in this note we mainly restrict our attention to the fertility, the same argument gives an information about more general notions; a knot $K$ is called \emph{$n$-fertile} if for any prime knot $K'$ with $c(K')\leq n$, there exists a minimum crossing number diagram $D$ of $K$ such that $K'$ is supported by $S(D)$. The \emph{fertility number} $F(K)$ is the maximum integer $n$ such that $K$ is $n$-fertile. 

\begin{theorem}
Let $K$ be a knot.
\begin{enumerate}
\item[(i)] $F(K) \leq c(K)-b(K)+3$. 
\item[(ii)] If $b(K) \leq 3$, $F(K) \leq \frac{4}{5}c(K)+6$. 
\end{enumerate}
\end{theorem}
\begin{proof}
By the same argument as Proposition \ref{prop:s(D)}, if $K$ is $m$-fertile then
$b(K) \leq c(K)-m+3$ if $m$ is even and $b(K) \leq c(K)-m+2$ if $m$ is odd. 
Similarly, when $b(K) \leq 3$ by the same argument as Theorem \ref{theorem:main}, if $K$ is $m$-fertile then $-3+\frac{3}{5}c(K) \leq c-\frac{m}{2}$ if $m$ is even and $-3+\frac{3}{5}c(K) \leq c - \frac{m+1}{2}$ if $m$ is odd.
\end{proof}

\section*{Acknowledgement}
The author has been partially supported by JSPS KAKENHI Grant Number 19K03490,16H02145. He is grateful to Ryo Hanaki whose talk at MSJ Autumn Meeting 2020 stimulates the author to study Question \ref{question:main}.


\begin{thebibliography}{1}


\bibitem[Be]{Be} D. Bennequin,
{\em Entrelacements et {\'e}quations de Pfaff,} 
Ast{\'e}risque, 107-108, (1983) 87-161.

\bibitem[CH+]{CH+}
J. Cantarella, A. Henrich, E. Magness, O. O'Keefe, K. Perez, E. J. Rawdon
and B. Zimmer,
{\em Knot fertility and lineage,}
J. Knot Theory Ramifications 26 (2017) 1750093.

\bibitem[BM]{BM} J. Birman and W. Menasco,   
{\em Studying links via closed braids. II. On a theorem of Bennequin. }Topology Appl. \textbf{40} (1991), no. 1, 71--82.

\bibitem[Ha]{Ha} R. Hanaki,   
{\em On Fertility of knot shadows}, preprint.

\end{thebibliography}
\end{document}